\documentclass{amsart}
\usepackage{amssymb,amsmath,amsthm,amsfonts,graphicx,epsfig}
\usepackage{paralist}
\usepackage[utf8]{inputenc}
\usepackage[T1]{fontenc}
 \usepackage[colorlinks=true]{hyperref}
\hypersetup{urlcolor=blue, citecolor=red}

\newtheorem{teo}{Theorem}[section]

\newtheorem{lemma}[teo]{Lemma}
\newtheorem{prop}{Proposition}
\theoremstyle{definition}
\newtheorem{df}[teo]{Definition}
\theoremstyle{remark}
\newtheorem{obs}[teo]{Remark}
\newtheorem{ejp}[teo]{Example}
\newtheorem{prob}[teo]{Problem}

\DeclareMathOperator{\len}{len}
\DeclareMathOperator{\diam}{diam}
\DeclareMathOperator{\clos}{clos}

\DeclareMathOperator{\dist}{dist}

\DeclareMathOperator{\fix}{Fix}
\newcommand{\expc}{\eta}
\newcommand{\clocks}{\mathcal C}

\newcommand{\R}    {\mathbb R}

\newcommand{\Z}  {\mathbb Z}

\renewcommand{\epsilon}{\varepsilon}

\title[Singular Cw-Expansive Flows]{Singular Cw-Expansive Flows}
\author[A. Artigue]{}

\email{artigue@unorte.edu.uy}
\begin{document}

\maketitle

\centerline{\scshape Alfonso Artigue }
\medskip
{\footnotesize
 \centerline{Departamento de Matemática y Estadística del Litoral }
   \centerline{Universidad de la Rep\'ublica}
   \centerline{Gral. Rivera 1350, Salto, Uruguay}
}

\bigskip

\begin{abstract}
We study continuum-wise expansive flows with fixed points on metric spaces and low dimensional manifolds.
We give sufficient conditions for a surface flow to be singular cw-expansive and examples showing that 
cw-expansivity does not imply expansivity. 
We also construct a singular Axiom A vector field on a three-manifold being singular cw-expansive 
and with a Lorenz attractor and a Lorenz repeller in its non-wandering set.
\end{abstract}

\section{Introduction}
In this paper we consider the dynamics of flows on compact metric spaces, in particular, three-manifolds and surfaces.
We consider singular flows with some kind of expansivity. 
If $\phi\colon\R\times M\to M$ is a flow of a metric space $(M,\dist)$ we say that $p\in M$ 
is a \emph{singular} point or a \emph{fixed} point if $\phi_t(p)=p$ for all $t\in \R$.
For the study of expansive flows with fixed points Komuro \cite{K} 
introduced the notion of $k^*$-expansivity.
In that paper he proved that the Lorenz attractor is $k^*$-expansive.
The concept of cw-expansivity for flows (see \S \ref{secCorCwExp} for the definition) was recently introduced by Cordeiro \cite{Willy}. 
In his Thesis, a deep analysis of the fundamental properties of such flows was developed. 
Among other things, he proved that the topological entropy of a cw-expansive flow on a compact metric space with topological 
dimension greater than 1 is positive. 
Also, the definition of $N$-expansive flow was introduced in this work. 
These concepts for flows extends the corresponding concepts of cw-expansivity \cite{Ka93} and $N$-expansivity \cite{Mo12} for homeomorphisms.

The purpose of the present article is the study of cw-expansive flows with fixed points. 
In \S \ref{secCwMet} general properties of cw-expansivity are proved for flows on metric spaces.
In Propositions \ref{propEqDfSingCwNoFix} and \ref{propKomuroCordeiro} we apply results from \cite{Willy} 
to prove equivalent definitions of cw-expansivity. 
From our viewpoint, these equivalences simplifies the understanding of these definitions 
because we reduce the number of parameters involved in the definition.
Following \cite{ArKinExp}, we also introduce a \emph{kinematic} version of cw-expansivity to prove some properties with greater generality.
In \S \ref{secSingSurf} we consider flows of surfaces.
In Proposition \ref{propCondSufCwExp} we give sufficient conditions for a surface flow to be singular cw-expansive. 
In \cite[Quest\~ao 2]{Willy} it is asked if there exists a singular cw-expansive flow with fixed points defined on a 
continuum that is not $k^*$-expansive.
Applying Proposition \ref{propCondSufCwExp} we show that such flows do exist.
The given examples present infinitely many fixed points, wandering points and, moreover, Lyapunov stable fixed points. 
As shown in \cite{Ar} these properties cannot be present on a $k^*$-expansive flow of a compact surface.
In \S \ref{secThreeMan} we construct singular cw-expansive vector fields on three-dimensional manifolds. 
The examples are singular Axiom A vector fields with two basic sets, namely, a Lorenz attractor and a Lorenz repeller.

\section{Cw-expansive flows}
\label{secCwMet}
In this section we recall the definitions related to cw-expansive flows, 
we prove some equivalences and we give some topological and dynamical properties of the set of fixed points of the flow. 
In our presentation we use Massera's terminology of \emph{clocks} \cite{Massera}, which means \emph{reparameterizations} or \emph{time changes}. 

\subsection{Cordeiro's cw-expansivity}
\label{secCorCwExp}
A \emph{clock} is an increasing homeomorphism $c\colon\R\to\R$ such that $c(0)=0$ 
and the \emph{clock-space} is the set of all clocks that will be denoted as $\clocks$. 
We consider the clock-space endowed with the compact-open topology. 
Let $(M,\dist)$ be a compact metric space. 
A \emph{continuum of clocks} is a continuous function $\alpha\colon A\to \clocks$ 
with $A\subset M$ a continuum. 
Recall that a \emph{continuum} is a compact connected subset.
For a continuum of clocks $\alpha$ denote as $\alpha_x$ the clock $\alpha(x)$.
Consider a flow $\phi\colon\R\times M\to M$ and 
define 
\[
 \Phi_t^\alpha(A)=\{\phi_{\alpha_x(t)}(x):x\in A\}.
\]
The continuity of $\alpha$ implies that $\Phi^\alpha_t(A)$ is a continuum for all $t\in\R$.
We say that $\phi$ is a \emph{cw-expansive} flow in the sense of Cordeiro \cite{Willy} if 
for all $\epsilon>0$ there is $\delta>0$ 
such that if $\alpha\colon A\to\clocks$ is a continuum of clocks 
and $\diam(\Phi_t^\alpha(A))<\delta$ for all $t\in\R$ then $A\subset \phi_{(-\epsilon,\epsilon)}(x)$ for some $x\in A$. 
In this case we say that $\delta$ is an \emph{expansivity constant}.
If $I\subset\R$ is an interval 
then $\phi_I(x)=\{\phi_t(x):t\in I\}$ is an \emph{orbit segment}. 
Note that an orbit segment may consist on an arc contained in an orbit, a fixed point or a circle (if the orbit is periodic and $I$ is sufficiently large).

Denote by $\fix(\phi)$ the set of fixed points of $\phi$.
In \cite[Lema 2.1]{Willy} it is shown that if a flow $\phi$ is cw-expansive then $\fix(\phi)$ is totally disconnected and open in $M$.
Note that given a cw-expansive flow $\phi$ on $M$ and a totally disconnected set $X$ disjoint from $M$ 
we obtain a cw-expansive flow on $M\cup X$ extending $\phi$ with fixed points at every $x\in X$.
This means that fixed points of a cw-expansive flow have no dynamical relevance (this also happens for Bowen-Walters expansivity \cite{BW}). 

In the next proposition we give an equivalent definition of cw-expansivity of flows that is a motivation for \S \ref{secSingCwExp} 
where we introduce the definition of singular cw-expansivity. 
For this purpose we cite the following result.

\begin{teo}[\cite{Willy}]
\label{teoCwExpEquiv}
A flow $\phi$ of a compact metric space is cw-expansive if and only if 
for all $\epsilon>0$ there is $\delta>0$ such that if $\alpha\colon A\to\clocks$ is a continuum of clocks
and $\diam(\Phi_t^\alpha(A))<\delta$ for all $t\in\R$ then $A$ is an orbit segment with $\diam(A)<\epsilon$. 
\end{teo}

In the definition of cw-expansive flow or its equivalent of Theorem \ref{teoCwExpEquiv} there are 
two parameters $\epsilon,\delta$.
The next result shows that cw-expansivity can be equivalently defined using only one of these parameters, obviously, the expansivity constant.

\begin{prop}
\label{propEqDfSingCwNoFix}
 A flow $\phi$ of a compact metric space without fixed points is cw-expansive if and only if 
 there is $\expc>0$ such that if $\alpha\colon A\to\clocks$ is a continuum of clocks
and $\diam(\Phi_t^\alpha(A))<\expc$ for all $t\in\R$ then $A$ is an orbit segment.
\end{prop}

\begin{proof}
The direct part follows by Theorem \ref{teoCwExpEquiv}. 
To prove the converse, for $\epsilon>0$ given define $\delta=\min\{\epsilon,\expc\}$, where $\expc$ is given in the hypothesis. 
Suppose that $\alpha\colon A\to\clocks$ is a continuum of clocks and $\diam(\Phi_t^\alpha(A))<\delta$ for all $t\in\R$. 
As $\delta\leq\expc$, we conclude that $A$ is an orbit segment. Taking $t=0$ we obtain $\diam(A)<\delta$. 
Given that $\delta\leq\epsilon$ we conclude that $A$ is an orbit segment of diameter smaller than $\epsilon$.
\end{proof}

\begin{obs}
 In spite of the simplicity of the proof of Proposition \ref{propEqDfSingCwNoFix}, 
 we think that the result is important because it gives a simpler definition. 
 It is natural to ask if this can be also done for expansive flows, in the sense of Bowen and Walters \cite{BW}. 
 An example showing that this is not the case was given in \cite[Example 2.24]{ArKinExp}.
\end{obs}

\subsection{Singular cw-expansive flows}
\label{secSingCwExp}
Motivated by the equivalence of Proposition \ref{propEqDfSingCwNoFix} we introduce the next definition.

\begin{df}
\label{dfSingCwExp}
We say that $\phi$ is a \emph{singular cw-expansive} flow if there is $\expc>0$ such that if $\alpha\colon A\to\clocks$ is a continuum of clocks
and $\diam(\Phi_t^\alpha(A))<\expc$ for all $t\in\R$ then $A$ is an orbit segment. 
In this case we say that $\expc$ is an \emph{expansivity constant}.
\end{df}

By Proposition \ref{propEqDfSingCwNoFix} the definition of singular cw-expansivity is equivalent with cw-expansivity if there are no fixed points of the flow.
This equivalence does not hold for flows with fixed points. 
This can be seen in the circle. It is easy to prove that a flow on the circle is singular cw-expansive if and only if 
it has finitely many fixed points (possibly, no fixed points).
In the next result we prove the equivalence between Cordeiro's version of $k^*$-expansivity and 
Definition \ref{dfSingCwExp}.

For $x\in M$ and $r>0$ denote by $B_r(x)$ the open ball of radius $r$ centered at $x$ and 
if $A\subset M$ define $B_r(A)=\cup_{x\in A} B_r(x)$.

\begin{prop}
\label{propKomuroCordeiro}
 The following statements are equivalent: 
 \begin{enumerate}
  \item  $\phi$ is singular cw-expansive,
  \item (Cordeiro-Komuro cw-expansivity) for all $\epsilon>0$ there is $\delta>0$ 
  such that if $\alpha\colon A\to \clocks$ is a continuum of clocks with 
  $\diam(\Phi_t^\alpha(A))<\delta$ for all $t\in\R$ then there are $t_0\in\R$ and $x\in A$ such that $\Phi_{t_0}^\alpha(A)\subset \phi_{(t_0-\epsilon,t_0+\epsilon)}(x)$.
 \end{enumerate}
\end{prop}

\begin{proof}
 To prove that 2 implies 1 notice that 
 $\Phi_{t_0}^\alpha(A)\subset \phi_{(t_0-\epsilon,t_0+\epsilon)}(x)$ implies that $A$ is an orbit segment. 
 Then $\expc=\delta$ (for an arbitrary $\epsilon>0$) is an expansivity constant.
 
To prove the converse assume that $\phi$ is singular cw-expansive with expansivity constant $\expc$. 
It is easy to see that if $\gamma\in(0,\expc)$ then for
all $x\notin \fix(\phi)$ there exists $t\in\R$ such that $\phi_t(x)\notin
B_\gamma(\fix(\phi))$.
Take $\epsilon>0$. 

In this paragraph we will show that there exists $\beta>0$ such that 
if $\dist(x,\fix(\phi))\geq \gamma$ and $\diam(\phi_{[0,s]}(x))<\beta$ then $|s|<\epsilon$. 
By contradiction, suppose that there exist $x_n\notin
B_\gamma(\fix(\phi))$ and $t_n>\epsilon$ such that
$\diam(\phi_{[0,t_n]}(x_n))\to 0$. Taking a subsequence we can
suppose that $x_n\to z$. Notice that $z$ is a regular point and then there exists $\epsilon'\in(0,\epsilon)$
such that $\phi_{\epsilon'}(z)\neq z$. By the continuity of $\phi$
we have that $\phi_{\epsilon'}(x_n)$ converges to
$\phi_{\epsilon'}(z)$. This is a contradiction because
$\diam(\phi_{[0,t_n]}(x_n))\to 0$.

To show that $\delta=\min\{\expc,\beta\}$ satisfies the thesis suppose that
$$\diam(\Phi_t^\alpha(A))<\delta$$ for all $t\in\R$ and some continuum of clocks $\alpha\colon A\to\clocks$.
Our hypothesis implies that $A$ is an orbit segment.
Without loss of generality we can suppose that $A\cap\fix(\phi)=\emptyset$. Take $x\in A$ and
$t_0\in\R$ such that $\phi_{t_0}(x)\notin B_\gamma(\fix(\phi))$.
Since
$\phi_{t_0}(x)\notin B_\gamma(\fix(\phi))$ 
and $\diam(\Phi_t^\alpha(A))<\beta$
we have that $\Phi_t^\alpha(A)\subset \phi_{(\alpha(t_0)-\epsilon,\alpha(t_0)+\epsilon)}(x)$. 
\end{proof}

\subsection{Kinematic cw-expansivity}
\label{secKinExp}
In analogy with the definitions considered in \cite{ArKinExp}
we say that a flow $\phi$ is \emph{kinematic cw-expansive} 
if there is $\expc>0$ such that if 
$\diam(\phi_t(A))<\expc$ for all $t\in\R$ and some continuum $A$ then $A$ is an orbit segment.
Considering the clock $c(t)=t$ for all $t\in\R$, the \emph{perfect clock} in Massera's \cite{Massera} terms, 
we see that singular cw-expansivity implies kinematic cw-expansivity.
In this section we will derive some properties of singular cw-expansive flows 
only assuming this weak form of expansivity.

\begin{prop}
\label{propKcw-expSingTotDisc}
 If $\phi$ is kinematic cw-expansive then $\fix(\phi)$ is totally disconnected. 
\end{prop}

\begin{proof}
 If $\fix(\phi)$ is not totally disconnected then there is a non-trivial continuum $A\subset \fix(\phi)$. 
 Moreover, given an arbitrary $\expc>0$ we can assume that $\diam(A)<\expc$. 
 Since $A\subset\fix(\phi)$ we have that $\phi_t(A)=A$ for all $t\in\R$, contradicting the cw-expansivity of $\phi$.
\end{proof}

A \emph{singular connection} is a regular orbit $\phi_\R(x)$ such that 
$\lim_{t\to+\infty}\phi_t(x)$ and $\lim_{t\to-\infty}\phi_t(x)$ are fixed points.

\begin{prop}
 If $\phi$ is kinematic cw-expansive with expansivity constant $\expc$ then 
 there are no singular connections of diameter smaller than $\expc$.
\end{prop}

\begin{proof}
 If $\phi_\R(x)$ is a singular connection of diameter smaller than $\expc$ then 
 consider the continuum $A=\clos(\phi_\R(x))$. 
 We have that $\phi_t(A)=A$ for all $t\in\R$ and $\diam(A)<\expc$. 
 This contradicts the hypothesis because $A$ is not an orbit segment.
\end{proof}

A closed $\phi$-invariant set $\Lambda\subset M$ is $\phi$-\emph{isolated} 
if there is an open set $U$ containing $\Lambda$ such that 
if $\phi_\R(x)\subset U$ then $\phi_\R(x)\subset\Lambda$.
Note that $\fix(\phi)$ is a closed $\phi$-invariant set. 
For a $k^*$-expansive flow, $\fix(\phi)$ is finite and isolated.
In the following example we show that without restrictions on the topology of 
$M$ the set $\fix(\phi)$ may not be $\phi$-isolated for a kinematic cw-expansive flow.

\begin{ejp}
\label{ejpNoFixIso}
 Consider the set 
 \[
  M=\{(x,y)\in\R^2:x^2+y^2=1/n, n\in\Z^+\}\cup\{(0,0)\}
 \]
 and a flow $\phi$ in $M$ with a periodic orbit in each circle $x^2+y^2=1/n$ and a fixed point in $(0,0)$.
 It is easy to prove that $\fix(\phi)=\{(0,0)\}$ is not $\phi$-isolated. 
 To see that the flow is cw-expansive note that every continuum $A\subset M$ is contained in an orbit, so, the proof is trivial.
\end{ejp}

In the flow given in Example \ref{ejpNoFixIso} we see that the fixed point is not isolated because 
it is accumulated by small periodic orbits. The next result shows that this is the only possible case.

\begin{prop}
\label{propFixNoIsoPer}
 If $\phi$ is a kinematic cw-expansive flow of the compact metric space $M$ 
 then there is a neighborhood $U$ of $\fix(\phi)$ such that if $\phi_\R(x)\subset U\setminus\fix(\phi)$ then $\phi_\R(x)$ is a periodic orbit.
\end{prop}

\begin{proof}
Let $\expc$ be an expansivity constant of the flow. 
By Proposition \ref{propKcw-expSingTotDisc} we know that $\fix(\phi)$ is totally disconnected. 
Then, there are $U_1,\dots,U_n$ disjoint open subsets of $M$ 
covering $\fix(\phi)$ such that $\diam(U_i)<\expc$ for each $i=1,\dots,n$. 
Define $U=U_1\cup\dots\cup U_n$.
Suppose that $\phi_\R(x)\subset U$ for some regular point $x$. 
Then, there is $i=1,\dots,n$ such that $\phi_\R(x)\subset U_i$. 
Consider the continuum $A=\clos(\phi_\R(x))$. 
Since $\phi_t(A)=A$ for all $t\in\R$ we conclude that $A$ is contained in an orbit. 
This implies that $A$ is the orbit of $x$, i.e., the orbit of $x$ is closed and 
$x$ is periodic.
\end{proof}

As we proved in Example \ref{ejpNoFixIso}, the set of fixed points may not be $\phi$-isolated for a 
kinematic cw-expansive flow on an arbitrary metric space. 
On a compact surface $\fix(\phi)$ is $\phi$-isolated.

\begin{prop}
 If $M$ is a compact surface and $\phi$ is a kinematic cw-expansive flow on $M$ 
 then $\fix(\phi)$ is $\phi$-isolated.
\end{prop}

\begin{proof}
By Proposition \ref{propFixNoIsoPer} we have that if $\fix(\phi)$ is not isolated then 
there are periodic orbits of arbitrarily small diameter. 
Take a periodic orbit bounding a $\phi$-invariant disc $A$. 
The continuum $A$ contradicts the kinematic cw-expansivity of the flow.
\end{proof}

\section{Singular flows of surfaces}
\label{secSingSurf}
In \cite{Ar} it is proved that the expansivity of surface flows depends on the geometric separating effect 
of the stable and unstable manifolds of the fixed points of saddle type.
In this section we establish, with a result and examples, that new mechanisms 
may produce cw-expansive flows of surfaces.

\begin{obs}
 The first thing to observe with respect to singular cw-expansive flows of surfaces is that they are singular, that is, 
 there must be at least one fixed point. 
 This can be proved following \cite{Ar,Flinn,LG}, where this result is proved for expansive flows.
\end{obs}

For the study of singular cw-expansive flows of surfaces we introduce a definition of separatrix following \cite{ABZ}. 
Given $\epsilon>0$ a \emph{positive} $\epsilon$-\emph{separatrix} of a flow $\phi$ is a positive trajectory $l^+=\phi_{\R^+}(y)$ 
such that $\lim_{t\to+\infty}\phi_t(y)$ is a fixed point and 
for all $\delta>0$ there is $z\in B_\delta(y)$ with $\phi_s(z)\notin B_\epsilon(l^+)$ for some $s>0$.
A \emph{negative} $\epsilon$-\emph{separatrix} is a positive $\epsilon$-separatrix for the inverse flow $\phi^{-1}_t=\phi_{-t}$. 
An $\epsilon$-\emph{separatrix} is a positive or negative $\epsilon$-separatrix.

\begin{prop} 
\label{propCondSufCwExp}
 If $\phi$ is a flow of a compact surface $M$ satisfying: 
 \begin{enumerate}
  \item $\fix(\phi)$ is totally disconnected and $\phi$-isolated, and 
  \item there is $\epsilon>0$ such that the union of the $\epsilon$-separatrices is dense in $M$
 \end{enumerate}
then $\phi$ is singular cw-expansive.
\end{prop}

\begin{proof}
Given $\epsilon$ as in the statement we will show that any $\expc\in(0,\epsilon)$
is an expansivity constant of the flow.
Let $\alpha\colon A\to\clocks$ be a continuum of clocks and
assume that $A$ is not an orbit segment.
Since $\fix(\phi)$ is totally disconnected the continuum $A$ contains a regular point $x$. 
Given that the $\epsilon$-separatrices are dense and $A$ is not contained in an orbit, 
there is a regular point $y\in A$ in an $\epsilon$-separatrix. 
Without loss of generality assume that $l^+=\phi_{\R^+}(y)$ is a positive $\epsilon$-separatrix. 
Also, we can assume that 
$\phi_{(-\tau,\tau)}(A)$ is a neighborhood of $y$ contained in $B_\expc(y)$, for some $\tau>0$. 
Let $\delta>0$ be such that $B_\delta(y)\subset \phi_{(-\tau,\tau)}(A)$.
By the definition of separatrix there is $z\in B_\delta(y)$ such that 
$\phi_s(z)\notin B_\expc(l^+)$, for some $s>0$.
Since $B_\delta(y)\subset \phi_{(-\tau,\tau)}(A)$, we can assume that $z\in A$. 
Take $t_0>0$ such that $\alpha_z(t_0)=s$. 
Then, as $y,z\in A$ and $\phi_s(z)\notin B_\expc(\phi_{\R^+}(y))$ we conclude that 
$\dist(\phi_{\alpha_y(t_0)}(y), \phi_{\alpha_z(t_0)}(z))>\expc$.
In this way $\diam(\Phi^\alpha_{t_0}(A))>\expc$ and the proof ends.
\end{proof}

\begin{prob}
Is the converse of Proposition \ref{propCondSufCwExp} true? 
The problem we found is the possibility of infinitely many fixed points (for instance, a Cantor set).
\end{prob}

\begin{ejp}
\label{ejpIrrSing}
Let $M$ be the two-dimensional torus and consider a vector field $X$ defining an irrational flow on $M$. 
Take a point $p\in M$ and a smooth map $\rho\colon M\to\R$ such that $\rho(p)=0$ and 
$\rho(x)>0$ for all $x\neq p$. 
Let $\phi$ be the flow induced by the vector field $\rho X$. 
This flow is singular cw-expansive but it is not $k^*$-expansive.
Let us give a variation of this example.
Given an arc $l\subset M$ transverse to $X$, a totally disconnected set $K\subset l$ and 
a non-negative smooth map $\rho\colon M\to \R$ vanishing only at $K$ 
the flow induced by $\rho X$ is cw-expansive but it is not $k^*$-expansive.
\end{ejp}

\begin{obs}
We recall that in \cite{Willy} it is shown that every non-singular cw-expansive flow 
has positive topological entropy if the space has topological dimension greater than 1.
Since surface flows have vanishing topological entropy \cite{Young}, we conclude that the flows in Example \ref{ejpIrrSing} 
are singular cw-expansive with vanishing topological entropy. 
\end{obs}

The following is an example that presents wandering points and a Lyapunov stable periodic orbit.

\begin{ejp}
\label{ejpEjpAnnulus}
 Let $f\colon [0,1]\to[0,1]$ be an increasing diffeomorphism such that 
 $f(x)=x/2$ if $x\in[0,1/2]$ 
 and $f(x)>x$ for all $x\in(0,1)$. 
 Let $X$ be a vector field on the annulus 
 \[
  M=\{(x,y)\in\R^2:1\leq\sqrt{(x-2)^2+y^2}\leq 4\}
 \]
 such that $l=[0,1]\times \{0\}\subset M$ is a global cross section 
 with Poincaré return map $(x,0)\mapsto (f(x),0)$, for all $x\in [0,1]$. 
 For each $n\geq 1$ consider a finite set $Y_n\subset [1/2^{n+2},1/2^{n+1}]$ 
 such that $Y=\cup_{n\geq 1} f^{-n}(Y_n)$ is dense in $[1/4,1/2]$.
Consider a non-negative function 
$\rho\colon M\to\R$ vanishing only at $\clos(Y)$. 
By Proposition \ref{propCondSufCwExp},
the flow induced by the vector field $\rho X$ is singular cw-expansive.
Notice that the periodic orbit contained in the smaller circle is Lyapunov stable for $t\to-\infty$. 
It is known that a non-singular cw-expansive flow of a locally connected, compact and connected metric space (of arbitrary dimension) 
cannot have Lyapunov stable points.
\end{ejp}

\begin{obs}
If we collapse the small circle in the boundary of the annulus $M$ of Example \ref{ejpEjpAnnulus} we obtain a disc, 
and it is easy to see that the flow (in the disc) is singular cw-expansive. 
Adding fixed points in the boundary of the disc, taking copies of the disc and gluing orbits in the boundaries 
we can obtain a singular cw-expansive flow on an arbitrary compact surface. 
Expansive flows in the sense of Komuro do not exist on every surface, see \cite{Ar}. 
The possibility of non-trivial recurrence is necessary in order that a surface admits a $k^*$-expansive flow.
\end{obs}

\section{Singular cw-expansivity on three-manifolds}
\label{secThreeMan}
In this section we will construct singular cw-expansive vector fields on three-manifolds.
Let $M$ be a three-dimensional compact smooth manifold without boundary. 
On $M$ we consider a smooth vector field $X$.
Following \cite{ArPa} we say that $X$ is \emph{singular Axiom A} 
if its non-wandering set is a finite union 
$\Omega(X)=\Omega_1\cup\dots\cup\Omega_k$ 
with each $\Omega_i$ being compact, invariant, transitive, isolated, with a dense set of periodic points and 
either uniformly hyperbolic (in the standard sense of Smale's Axiom A) or a singular-hyperbolic attractor or a repeller.
For the next result we need to know that the Lorenz attractor is singular hyperbolic \cite{ArPa}.

\begin{teo}
\label{teoSingAxA}
 There are singular Axiom A singular cw-expansive vector fields on three-manifolds
 with $\Omega(X)=A\cup R$ where 
 $A$ is a Lorenz attractor and $R$ is a Lorenz repeller.
\end{teo}

\begin{proof}
We start with a sketch of the construction.
Consider a Lorenz attractor $A$ in a three-dimensional Euclidean space $V$ 
defined by a vector field $X\colon V\to TV$, where $TV$ denotes 
the tangent bundle of $V$. 
Let $S\subset V$ be a smooth genus-two surface transverse to $X$ bounding the attractor $A$. 
Also assume that if $U$ is the open neighborhood of $A$ with $\partial U=S$ then 
$A=\cap_{t\geq 0} X_t(U)$.

Let $V'$ be a vector space isomorphic to $V$ and its corresponding vector field $X'$. 
If we consider the vector field $-X'$ we obtain a Lorenz repeller $R\subset V'$.
Analogously, consider a genus-two surface $S'$ and an open neighborhood $U'$ of $R$. 

Given a diffeomorphism $\varphi\colon S\to S'$ we consider the closed 
three-manifold 
\[
 M=\frac{U\cup S\cup U'\cup S'}{x\sim\varphi(x)}.
\]
The manifold $M$, whose topology depends on the gluing map $\varphi$, 
is the phase space of our examples. 
The global dynamic is \emph{simple}, there is an attractor, a repeller and 
wandering trajectories goes from the repeller to the attractor meeting the transverse surface $S$ in exactly one point. 

The expansivity on the non-wandering set was shown by Komuro ($k^*$-expansivity) \cite{K}. 
If $x$ is wandering and $y$ is non-wandering then for some $t$ we have that $\phi_t(x)$ is in $S$, while 
the whole orbit of $y$ is in the non-wandering set. Then $x$ and $y$ are separated by the flow. 
This reduces the problem of proving any kind of expansivity (cw-expansivity or $N$-expansivity) to study wandering points. 
Since every wandering trajectory meets the transverse surface $S$, we have to understand 
the stable and the unstable foliations restricted to $S$.

We now focus our attention to the attractor $A$. 
Assume that $A\subset \R^3$.
We will study the stable foliation via the geometric model of the Lorenz attractor. 
Following \cite{Robinson1981}, this model presents a hyperbolic fixed point 
$\sigma=(0,0,0)$ in $\R^3$ with $\dim(W^s(\sigma))=2$ and $\dim(W^u(\sigma))=1$.  
The classical picture of the geometric model is given in Figure \ref{figClassModel}.
The boundary of this volume is a genus two surface but it is not transverse to the flow. 
Also, it does not contain the attractor in its interior, for example, the fixed point $\sigma$ is 
in the boundary.

\begin{figure}[h]
\includegraphics{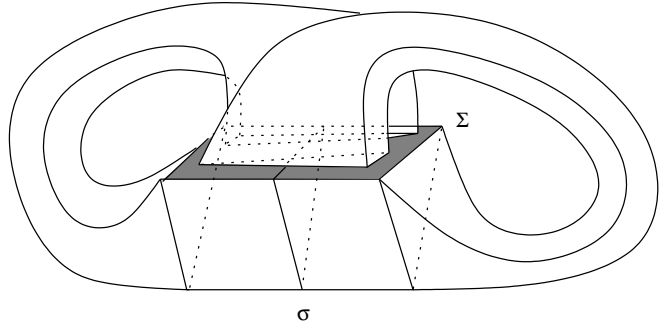} 
\caption{Classical geometric model of the Lorenz attractor.}
\label{figClassModel}
\end{figure}

To obtain our transverse surface $S$ we need to consider a neighborhood of the classical geometric model. 
For this purpose we introduce the following \emph{topological model} for the Lorenz attractor. 
See Figure \ref{figTopoModel}. 

\begin{figure}[h]
\includegraphics{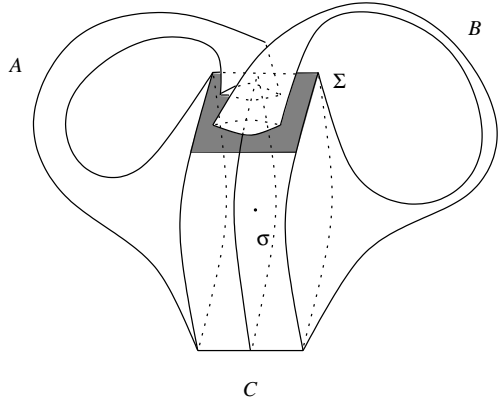} 
\caption{Topological model of the Lorenz attractor.}
\label{figTopoModel}
\end{figure}

We assume that in both models the cross section $\Sigma$ is given by 
\[
 \Sigma=\{(x,y,1):|x|,|y|\leq 1\}
\]
and that the stable foliation restricted to $\Sigma$ is given by constant $x$ coordinate.
The topological model consists on three topological cylinders $A,B,C$ shown in Figure \ref{figTopoModel}. 
The flow enters through the curved face of the cylinder $C$. 
The cylinders $A$ and $B$ are flow tubes (the trajectories are tangent in the boundary of $A$ and $B$). 
To plainly understand how to construct the (genus two) transverse surface $S$ 
consider the following version of the tube $B$ where we have taken a copy of $\Sigma$. 
The new cylinder that replaces the tube $B$ is called $B'$. See Figure \ref{figTubeB}.

\begin{figure}[h]
\includegraphics{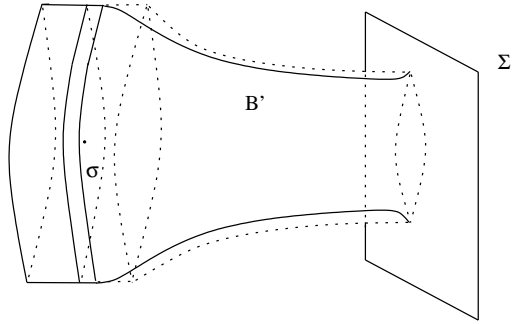} 
\caption{The cylinder $B'$ is transverse to the flow.}
\label{figTubeB}
\end{figure}

Analogously, it is constructed a cylinder $A'$ replacing the flow tube $A$. 
In this way we obtain our surface $S$, shown in Figure \ref{figTopoLor}. 

Now we must understand how does the stable foliation look when restricted to $S$. 
Take a point $x\in S$. 
If for some $t>0$ we have that $\phi_t(x)\in\Sigma$ then 
the stable foliation on a neighborhood of $x$ (a neighborhood in $S$) 
is obtained as the pull-back of the stable foliation on $\Sigma$. 
At these points we obtain a smooth foliation. 
We know that $S$, having genus two, 
must have singular points of the stable foliation. 
These singular points appear at the intersection of the stable manifold of $\sigma$ with $S$. 
Notice that there is an arc in $S$, disjoint from $\Sigma$, of points whose positive trajectory 
goes to $\sigma$ without crossing $\Sigma$.
Denote by $W^{ss}(\sigma)$ the strong stable manifold (one-dimensional) of $\sigma$.

\begin{figure}[h]
\includegraphics{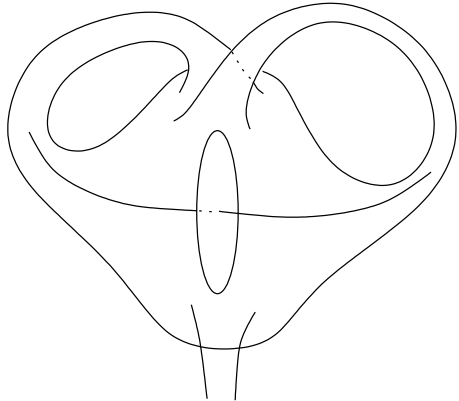} 
\caption{The genus two surface $S$ transverse to the flow and containing the Lorenz attractor.}
\label{figTopoLor}
\end{figure}

 Now we show that the stable foliation restricted to $S$ has two singular points of saddle type at 
 $W^{ss}(\sigma)\cap S$.
 Consider the transverse sections shown in Figure \ref{figCrossSections}. 
The rectangle $E$ is contained in $\Sigma$ where the stable foliation is parallel to $W^{ss}(\sigma)$ in the linear neighborhood of $\sigma$. 
The rectangle $F$ is perpendicular to $W^u(\sigma)$ and the stable foliation is parallel to $W^{ss}(\sigma)$ too. This is 
a usual assumption in the geometric model. 
The rectangle $G$ is perpendicular to $W^{ss}(\sigma)$. 
The linearity of the model near $\sigma$ allows us to explicitly calculate the flow and the Poincaré maps. 
This gives us that the stable foliation on $G$ has a singular point of saddle type. 
Analogous considerations gives us another singular point in the other branch of $W^{ss}(\sigma)$. 
\begin{figure}[h]
\includegraphics{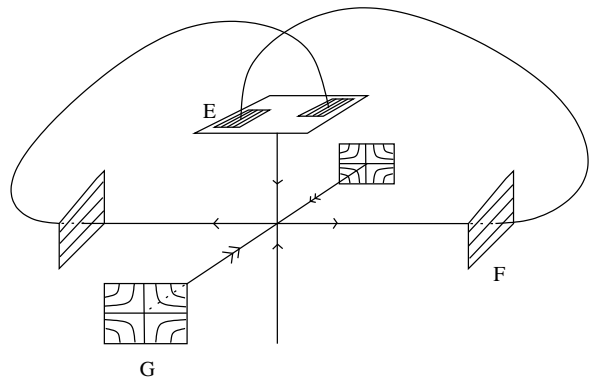} 
\caption{A singular point of the stable foliation appears in $G$.}
\label{figCrossSections}
\end{figure}

To finish the proof we note that we can choose the gluing map $\varphi$ in such a way that 
each stable leaf intersects each unstable leaf in a totally disconnected set, see Lemma \ref{lemPert} below. 
This implies that the flow is singular cw-expansive.
\end{proof}

\begin{lemma}
 \label{lemPert}
 Let $S,S'$ be two diffeomorphic compact surfaces with $C^1$ foliations
$F,F'$ with finitely many hyperbolic 
singular points. 
Then, there is a residual set of $C^1$ diffeomorphisms $\varphi\colon S\to S'$ 
such that $\varphi(L)\cap L'$ is totally disconnected
for every pair of leaves $L,L'$ of $F,F'$ respectively.
\end{lemma}

\begin{proof}
On the surface $S'$ consider a Riemannian metric and denote by $\len(\alpha)$ the length 
of a piecewise $C^1$ curve $\alpha$. 
Since $F'$ has hyperbolic singular points, there is $\epsilon_0>0$ such that 
no closed leaf has length smaller than $\epsilon_0$. 
For $\epsilon\in (0,\epsilon_0)$ let $D_\epsilon$ be the set of diffeomorphisms $\varphi\colon S\to S'$ 
such that if a curve $\alpha$ is contained in $\varphi(L)\cap L'$ then $\len(\alpha)<\epsilon$.

We will show that each $D_\epsilon$ is open and dense in the space of all diffeomorphisms from $S$ to $S'$. 
Then the residual set will be obtained as $\cap_{n\geq 1/\epsilon_0} D_{1/n}$, $n\in \Z$. 
To prove that $D_\epsilon$ is open we show that its complement is closed. 
Let $\varphi_k\notin D_\epsilon$ be a convergent sequence of diffeomorphisms, $\varphi_k\to \varphi$, such that 
for each $k\geq 1$ there is an arc $\alpha_k\subset \varphi_k(L_k)\cap L'_k$ for some sequences of leaves $L_k\subset S$ and $L'\subset S'$ and $\len(\alpha_k)=\epsilon$.
Take $x_k\in\alpha_k$ and assume that $x_k\to x\in S'$.
Since there are no closed leaves of length smaller than $\epsilon$, the limit of 
$\alpha_k$ in the Hausdorff metric is an arc of length $\epsilon$ contained in the leaves through $x$ of both foliations 
$F'$ and $\varphi(F)$. Then $\varphi\notin D_\epsilon$ and $D_\epsilon$ is open.

To prove that $D_\epsilon$ is dense consider $l_1,\dots,l_j\subset S'$ a finite number of smooth compact arcs transverse to the leaves of $F'$ such that each arc of a leaf of $F'$ of length $\epsilon$ intersects at least one of the arcs $l_i$. 
In addition we assume that the arcs $l_i$ are pairwise disjoint.
Given a diffeomorphism $\varphi\colon S\to S'$ we will show that there is a small $C^1$ perturbation of $\varphi$ in $D_\epsilon$. 
We say that there is a $\varphi$-\emph{tangency} at $x\in S'$ if the leaves of $F'$ and $\varphi(F)$ are tangent at $x$.
First, with a small perturbation $\varphi_1$ of $\varphi$ we can assume that the set of $\varphi_1$-tangency points $x\in\cup_{i=1}^jl_i$ is finite.
Second, with a small perturbation $\varphi_2$ of $\varphi_1$ we can obtain that 
at the finite set of $\varphi_2$-tangencies there is no non-trivial arc simultaneously contained in leaves of both foliations $F'$ and $\varphi_2(F)$. These perturbations can be obtained with standard techniques of bump functions. 
Then, as every arc of a leaf of $F'$ of length $\epsilon$ intersects at least one of the $l_i$, we have that no arc of $\varphi_2$-tangencies has length $\epsilon$. This proves that $D_\epsilon$ is dense. 
\end{proof}

\begin{obs}
 In the rectangle $G$ of Figure \ref{figCrossSections} consider the central point, say $p$. 
 The local stable set of $p$ is just the vertical segment. 
 The trajectories of the points of the horizontal line of $p$ in $G$ are separated from the trajectory of $p$ 
 by the flow (for example, some of these trajectories goes to the rectangle $F$).
\end{obs}

\begin{obs}
 In \cite{ArS3} it is shown that the three-sphere admits $k^*$-expansive flows with hyperbolic fixed points, in particular, these flows are singular cw-expansive. 
 From the proof of Theorem \ref{teoSingAxA} we have that the three-sphere admits singular cw-expansive flows with wandering points. 
 For this purpose one has to consider a gluing diffeomorphism $\varphi$ in order to obtain the three-sphere and then to perform a small perturbation of Lemma \ref{lemPert}.
\end{obs} 
 
 \begin{obs}There are well known examples of Anosov flows on three-manifolds with wandering points \cite{FrWi}, which are robustly expansive (because Anosov flows are expansive and structurally stable \cite{Anosov}). 
 General properties of non-singular expansive flows of three-manifolds can be found in \cite{Pat}.
 Non-singular cw-expansive flows (not being expansive) of three-manifolds with wandering points can be obtained as the suspension of the examples in \cite{ArRobNexp}. 
 These examples are robustly (non-singular) cw-expansive in the $C^r$ topology with $r\geq 2$. 
\end{obs}

In \cite{Willy} it is introduced a definition of $N$-expansivity for flows. 
We will consider a variation of this definition in order to allow fixed points.

\begin{df}
Given a positive integer $N$, a flow $\phi$ of a metric space is \emph{singular} $N$-\emph{expansive} if 
 for all $\epsilon>0$ there is $\expc>0$ such that if 
 $A\subset M$ is compact, $\alpha\colon A\to\clocks$ is continuous and
 $\diam(\Phi^\alpha_t(A))<\expc$ for all $t\in\R$ then 
 $A$ is contained in $N$ orbit segments of diameter $\epsilon$.
\end{df}

\begin{obs}
From the construction of the flow in the proof of Theorem \ref{teoSingAxA} it seems that we can obtain 2-expansive flows. 
For this purpose it could be convenient to considered a gluing map $\varphi$ only introduce quadratic tangencies. 
\end{obs}

\begin{prob}
 Is there a gluing map $\varphi$ making the flow singular expansive in the sense of Komuro ($k^*$-expansive)? 
 If this is the case, can $M$ be the three-sphere?
 For an affirmative answer, it seems that the strong invariant manifolds of the fixed points (of the attractor and the repeller) must be connected (identified) by $\varphi$. 
\end{prob}

\begin{prob}
 With the techniques of the proof of Theorem \ref{teoSingAxA}, 
 is it possible to obtain a robust singular cw-expansive (or $N$-expansive) vector field? 
 A $C^r$ topology should be considered with $r\geq 2$.
 An idea, for a positive answer, could be to control the tangencies appearing at wandering points as for example in \cite{ArRobNexp}. 
 The tangencies, that necessarily will appear near the singularities of the stable foliation, must be quadratic. 
 It seems that the hardest part is to control the dependence of the stable foliation on the cross section $G$ with respect to the perturbation of the vector field. 
 It could be a known topic of \emph{perturbations} and \emph{invariant manifolds} theories, but we were not able to 
 find a solution in the literature and the details are not clear to the author.
\end{prob}

\medskip
\medskip
\end{document}